\documentclass[11pt]{article}
\usepackage{color}
\definecolor{darkred}{rgb}{0.8,0,0}

\usepackage{amsmath,amsthm,amsfonts,amssymb}

\usepackage{graphicx}

\newcommand{\captionfonts}{\footnotesize}
\makeatletter  
\long\def\@makecaption#1#2{%
  \vskip\abovecaptionskip
  \sbox\@tempboxa{{\captionfonts #1: #2}}%
  \ifdim \wd\@tempboxa >\hsize
    {\captionfonts #1: #2\par}
  \else
    \hbox to\hsize{\hfil\box\@tempboxa\hfil}%
  \fi
  \vskip\belowcaptionskip}
\makeatother   

\newcommand{\nwc}{\newcommand}

\newtheorem{prop}{Proposition}[section]
\newtheorem{lemma}[prop]{Lemma}

\newtheorem{theorem}[prop]{Theorem}
\newtheorem{corollary}[prop]{Corollary}
\newtheorem{definition}[prop]{Definition}

\nwc{\R}{\mathbb R}
\nwc{\Z}{\mathbb Z}
\nwc{\N}{\mathbb N}

\newcommand{\ignore}[1]{}

\nwc{\eps}{\varepsilon}
\nwc{\re}{Re\,}

\nwc{\wto}{\rightharpoonup}

\nwc{\ds}{\displaystyle}
\newcommand {\bedis} {\begin{displaymath}}
\newcommand {\edis} {\end{displaymath}}
\newcommand{\newbeqna} {\renewcommand {\arraystretch} {2}
                        \begin {displaymath} \begin {array}{crcl}}
\newcommand{\neweqna}{\end{array} \end {displaymath}}

\newcommand{\fbeqna}{\renewcommand {\arraystretch} {1.3}
\begin {displaymath}\begin{array}{rcll}}
\newcommand{\feqna}{\end{array}\end{displaymath}}

\newcommand {\beqna} {\begin{eqnarray*}}
\newcommand {\eqna} {\end{eqnarray*}}
\newcommand {\beqn} {\begin{eqnarray}}
\newcommand {\eqn} {\end{eqnarray}}

\newcommand{\p}{\varphi}

\begin{document}
\title{Oscillatory dynamics in 
Smoluchowski's coagulation equation with 
diagonal kernel
}

\author{
Ph. Lauren\c{c}ot
\and
B. Niethammer%
\and
J. J. L. Vel\'{a}zquez
}
\maketitle

\begin{abstract}

We characterize the long-time behaviour of solutions to Smoluchowski's coagulation
equation with a diagonal kernel of homogeneity $\gamma < 1$.
Due to the property of the diagonal kernel, the value of a solution at a given cluster size depends only on 
a discrete set of points. As a consequence, the long-time behaviour of solutions
is in general periodic, oscillating between different rescaled versions of a self-similar solution.
Immediate consequences of our result are a characterization of the set of data for which the solution
converges to self-similar form and a uniqueness result for self-similar profiles.

\end{abstract}

\section{Introduction}
\label{S.intro}

Smoluchowski's classical coagulation equation provides a mean-field description of
binary coalescence of clusters. 
If $\xi$ denotes the size of a cluster and $F(\tau,\xi)$ the corresponding number density
at time $\tau$  then the equation is given by
\begin{equation}\label{Smolu1}
 \begin{split}
\partial_{\tau}F(\tau,\xi) & =\frac{1}{2}\int_{0}^{\xi}K(\xi{-}\eta,\eta)  F(\tau,\xi{-}\eta)
F( \tau,\eta)  d\eta\\
&\qquad -\int_{0}^{\infty}K(  \xi,\eta)  F(\tau,\xi)  F(\tau,\eta)  d\eta\,, \qquad (\tau,\xi)\in (0,\infty)^2 \,,
\end{split}
\end{equation}
where the coagulation kernel $K=K(\xi,\eta)\ge 0$ describes the rate at which clusters of size $\xi$ and $\eta$ coagulate.
This model has been used in various applications, most prominently in aerosol physics or
polymerization, but also in astrophysics or population dynamics \cite{Aldous99, Drake72, Ram00}.

We are in the following interested in kernels that grow at most sublinearly at infinity. In this case it is well-known for 
a large class of kernels that, if the initial condition $F(0)$ has finite mass, i.e. finite first moment, then the solution $F$ to \eqref{Smolu1} conserves the mass for all positive times, that is
\begin{equation}\label{S1E2}
 \int_0^{\infty} \xi F(\tau,\xi)\,d\xi = \int_0^{\infty} \xi F(0,\xi)\,d\xi\, \qquad \mbox{ for all } \tau >0\,.
\end{equation}

A topic of  particular interest in the theory of coagulation is whether  solutions to \eqref{Smolu1}
evolve for large times in a self-similar fashion. This is however only well-understood for the so-called solvable
kernels $K(\xi,\eta)=const., K(\xi,\eta)=\xi+\eta$ and $K(\xi,\eta)=\xi\eta$, since in this case equation \eqref{Smolu1} can be solved explicitly by transform methods. A complete characterization of the long-time behaviour of solutions to \eqref{Smolu1} with solvable kernels can be found in \cite{MePe04}.

For all other kernels there are so far only numerical studies available \cite{FilLau04, Lee01} that suggest for a range of kernels
convergence to self-similar form in the long-time limit. In subsequent years several results established also rigorously the existence of self-similar
solutions, both with finite mass and with fat tails \cite{EMR05, FouLau05, Leyvraz05, NTV16, NV12a}, but an analysis of the long-time behaviour of solutions is still elusive.

In this paper we present the first rigorous analysis of the long-time behaviour of solutions for a non-solvable kernel. More precisely, we study the so-called diagonal kernel,
that is  in some sense orthogonal to the constant one. It is  given
 by
$$
K\left(  \xi,\eta\right)  =\xi^{1+\gamma}\delta\left(  \xi{-}\eta\right)\ , \qquad (\xi,\eta)\in (0,\infty)^2\, 
\ ,\qquad \gamma<1\,,
$$
where $\delta$ is the Dirac mass at zero, that is, in this model, only clusters of the same size can coagulate. 
Then  equation \eqref{Smolu1} reduces to
\begin{equation}
\partial_{\tau}F\left(  \tau,\xi\right)  =\frac{\xi^{1+\gamma}}{2^{3+\gamma}%
}  F^2\left(  \tau,\xi/2\right) -\xi^{1+\gamma
} F^2\left(  \tau,\xi\right) \,, \qquad (\tau,\xi)\in (0,\infty)^2\,\,.\label{S1E1}%
\end{equation}
Obviously a key difference to continuous kernels is the fact that, for the diagonal one, the evolution of the number density $F$ of clusters with mass $\xi$ depends only on smaller values and not on larger ones. As a consequence, the equation obeys a maximum principle, a fact that we will also use in our analysis.
Another key feature is that the evolution in $\xi$ depends only on a discrete set of values of the form $\xi/2^k$ and thus the evolution decouples
in a certain sense which leads to a long-time behaviour of solutions that is different from what one expects for continuous kernels.
We will explain this in more detail in Section~\ref{Ss.results} below.

\section{Main results}\label{S.results}

\subsection{Reformulation of the problem}

It turns out to be more convenient to go over to the function
\[
 G(\tau,\eta)=2^{\eta(1{+}\gamma)} F(\tau,\xi) \qquad \mbox{ with } \xi=2^{\eta}\,,\; \eta \in \R\ , \mbox{ and } \tau>0\,.
\]
Then $G$ satisfies
\begin{equation}\label{Gequation}
 \partial_{\tau} G(\tau,\eta) = \frac{1}{2^{1{-}\gamma}} G^2(\tau,\eta{-}1) - G^2(\tau,\eta)\,, \qquad (\tau,\eta)\in (0,\infty)\times\R\,,
\end{equation}
and
\begin{equation}\label{Gmass}
 \int_{\R} 2^{(1{-}\gamma)\eta} G(\tau,\eta)\,d\eta = \int_{\R} 2^{(1{-}\gamma)\eta} G(0,\eta)\,d\eta\,\qquad \mbox{ for all } \tau >0\,.
\end{equation}
Well-posedness of \eqref{Gequation} for initial data in $L^{\infty}(\R)$ will be shown later for a reformulated problem in Lemma~\ref{L.wp}.

We expect to have as special solutions to \eqref{S1E1} self-similar solutions of the form 
$$
(\tau,\xi)\longmapsto \frac{1}{\tau^{2/(1-\gamma)}} \phi\left( \frac{\xi}{\tau^{1/(1-\gamma)}} \right)\,, \qquad (\tau,\xi)\in (0,\infty)^2\,,
$$
which means for \eqref{Gequation} rescaled traveling wave solutions of the form
$$
G(\tau,\eta)= \frac{1}{\tau} H\left( \eta - \frac{\ln \tau}{(1{-}\gamma) \ln 2} \right)\,, \qquad (\tau,\eta)\in (0,\infty)\times\R\,.
$$ 
In order to study the existence of such special solutions and the large time behaviour of solutions to \eqref{Gequation} we introduce the new variables
\begin{equation}\label{variables}
 \alpha:= (1{-}\gamma) \ln 2\,, \quad t:=\frac{\ln \tau}{\alpha} \,, \quad x:=\eta-\frac{\ln \tau}{\alpha}\,, \quad \mbox{ and } \quad e^{-\alpha t} h(t,x) := \alpha G(\tau,\eta)\,,
\end{equation}
and we deduce from \eqref{Gequation} and \eqref{Gmass} that $h$ solves
\begin{equation}\label{heq1}
\partial_t h(t,x) =  \partial_{x} h(t,x) + \alpha h(t,x) + e^{-\alpha} h^2(t,x{-}1) - h^2(t,x)\,, \qquad (t,x)\in\R^2\,,
\end{equation}
and
\begin{equation}\label{heq2}
 \int_{\R} e^{\alpha x} h(t,x)\,dx= \int_{\R} e^{\alpha x}  h(0,x)\,dx\,, \quad t\in\R \,.
\end{equation}

A self-similar solution to \eqref{S1E1} corresponds now to a stationary solution $\bar h$ of \eqref{heq1}, that is, to a solution of 
\begin{equation}\label{heq3}
0 =  \partial_{x}\bar h(x) + \alpha \bar h(x) + e^{-\alpha} {\bar h}^2(x{-}1) - {\bar h}^2(x)\,, \qquad x\in\R\,.
\end{equation}
Notice that the equation~\eqref{heq3} is invariant under a shift, that is, if $\bar h$ is a solution to \eqref{heq3}, then so is $x\mapsto \bar h(x+\lambda)$ for all $\lambda\in\R$. Recall that we are interested in solutions $\bar h$ to \eqref{heq3} that are nonnegative, locally integrable, and such that $\bar h \in L^1(\R; e^{\alpha x} \,dx)$. By shifting $\bar h$ we can normalize the mass so that 
$$
\int_{\R} e^{\alpha x} \bar h(x)\,dx = 1\,.
$$ 
Notice also that it follows directly from \eqref{heq3} and the local integrability of $\bar h$ that $\bar h \in C^{1}(\R)$.

It is proved in \cite{Leyvraz05} by a shooting method that for any $a>0$ there exists a nonnegative solution $\bar h_a\in L^1(\R; e^{\alpha x} \,dx)$ to \eqref{heq3} that is bounded, decreasing and satisfies
\begin{equation}\label{S2E4}
 \lim_{x \to -\infty} 2^{-\sigma x} \Big( \bar h_a(x)- \frac{\alpha}{1{-}e^{-\alpha}}\Big) = -a \,,
\end{equation}
where $\sigma$ is the only positive root of the equation $(1+\ln 2 \sigma/\alpha)(1-e^{-\alpha}) = 2(1-e^{-\alpha} 2^{-\sigma})$.
The constant $a$ is just a normalization and corresponds to rescaling the mass $\int_{\R}e^{\alpha x} \bar h_a(x)\,dx$. 
Furthermore, the function $\bar h_a$ satisfies 
\begin{equation}\label{S2E5}
 0<\bar h_a(x) \leq C_a\exp\Big( -L_a 2^x\Big) \qquad \mbox{ for some } \; C_a>0 \;\mbox{ and }\; L_a>0\,.
\end{equation}
Our goal in this paper is to study the behaviour of  solutions to \eqref{heq1}  as $t \to \infty$ and to figure out what is the role played by the family $(\bar h_a)_{a>0}$, if any. As a consequence of our results we will actually show that stationary solutions are unique up to rescaling.

\subsection{Main results}\label{Ss.results}

A key property of the diagonal kernel is that the evolution of $h(t,x)$ depends only on the discrete set of values $x-k$, $k \in \N$.
Correspondingly, one  of the key points of the argument used in this paper is a decomposition of
the plane $\left\{  \left(  t,x\right) \in\R^2\right\}$ in  a family of lines whose evolution under
(\ref{heq1}) is mutually decoupled.

More specifically, given $\theta\in [0,1)$, we define the following family of lines:
$$
\mathcal{S}_{\theta} := \left\{  \left(  t,x\right) \in\R^2 \,:\, x+t-\theta \in \Z \right\}\,.
$$
Obviously for given $\theta_{1},\theta_{2}\in [0,1)$ with $\theta_{1}\neq\theta_{2}$ we have $\mathcal{S}_{\theta_{1}} \cap\mathcal{S}_{\theta_{2}}=\emptyset$,
while $\bigcup_{\theta\in\left[  0,1\right) }\mathcal{S}_{\theta}=\R^2$.
Hence, the following function $\Theta$ is well-defined.

\begin{definition} \label{theta}
We define the function $\Theta\colon\mathbb{R}\times\mathbb{R}\rightarrow
\left[ 0,1\right)  $ as the unique $\theta=\Theta\left(  t,x\right)  $ such
that $\left(  t,x\right)  \in\mathcal{S}_{\Theta\left(  t,x\right)  }$, that is, $\Theta(t,x)=t+x -\lfloor t+x\rfloor$. 
\end{definition}

Notice that the function  $\Theta$ is $1$-periodic in both variables, that is
\begin{equation}
\Theta\left(  t+1  ,x\right)
=\Theta\left(  t,x\right)  \  ,\qquad \Theta\left(  t,x\right)  =\Theta\left(
t,x-1\right)  \  ,\qquad \left(  t,x\right)  \in\mathbb{R}
\times\mathbb{R}\,.\label{S3E6b}%
\end{equation}
We also define the function
\begin{equation}\label{psidef}
 \psi(t,\theta) := t-\theta -\lfloor t-\theta\rfloor \in [0,1)\,, \qquad (t,\theta) \in [0,\infty)\times [0,1)\,.
\end{equation}
Notice that in particular $\psi(0,\theta)=1-\theta$ for $\theta\in (0,1)$ and that $t\mapsto \psi(t,\theta)$ is right-continuous and jumps  from $1$ to $0$ at times $n+\theta$, $n \in \N$, 
\begin{equation}
\psi((n+\theta)^-,\theta) := \lim_{t\nearrow n+\theta} \psi(t,\theta) =1 \quad \mbox{ and } \quad \psi(n+\theta,\theta)=0\,. \label{jumppsi}
\end{equation}

\bigskip

We now fix a nonnegative solution $\bar h\in C^1(\R)\cap L^1(\R;e^{\alpha x}\,dx)$  to \eqref{heq3} enjoying the properties \eqref{S2E4} and \eqref{S2E5} and with mass normalized to one, that is 
\begin{equation}\label{heq4}
 1 = \int_{\R} e^{\alpha x} \bar h(x)\,dx= \int_0^1 \Big\{ \sum_{k\in \Z} e^{\alpha(k+\theta)} \bar h(k+\theta)\Big\}\,d\theta\,.
\end{equation}
Introducing 
$$
\nu(\theta) := \sum_{k \in \Z} e^{\alpha(k+\theta)} \bar h(k+\theta)\,, \quad \theta\in [0,1)\,,
$$
we infer from \eqref{heq3} and \eqref{S2E5} that 
$$
0 = \sum_{k\in\Z} \left[ \partial_x \bar h(k+\theta) + \alpha \bar h(k+\theta) \right] e^{\alpha(k+\theta)} = \sum_{k\in\Z} \frac{d}{d\theta} \left[ e^{\alpha(k+\theta)} \bar h(k+\theta) \right] = \frac{d\nu}{d\theta}(\theta)\,.
$$
Consequently, $\nu(\theta)=const.$ for $\theta\in [0,1)$ and the normalization \eqref{heq4} entails that 
\begin{equation}
\sum_{k\in\Z} e^{\alpha(k+\theta)} \bar h(k+\theta) = 1\ , \quad \theta\in [0,1)\,. \label{phl1}
\end{equation}
Then, in order to obtain from $\bar h$ a stationary solution to \eqref{heq1} with mass $M>0$, we need to shift $\bar h$ and see that the shift $\lambda$ is determined by the following relation:
\begin{equation}\label{phl2}
M = \sum_{k \in \Z} e^{\alpha k} \bar h(k-\lambda) \qquad \mbox{ if and only if }\qquad \lambda =  \frac{\ln M}{\alpha}\,.
\end{equation}

As pointed out above, the evolution of \eqref{heq1} decouples into an evolution for each 'fibre' $\mathcal{S}_\theta$, $\theta \in [0,1)$.  We shall see that, given a nonnegative initial condition $h_0\in L^\infty(\R)$, the long-time behaviour of the corresponding solution to \eqref{heq1} is determined by the mass 
of the initial condition in each fibre, given by 
\begin{equation}\label{mnoddef}
 m_0(\theta):= \sum_{k \in \Z} e^{\alpha ( k+\theta)} h_0(k+\theta)\,, \quad \theta\in [0,1)\,.
\end{equation}
More precisely, we will see that a solution of \eqref{heq1} is for large times approximated in each fibre $\theta\in [0,1)$ by the stationary solution with the same mass, that is, according to \eqref{phl2}, the shift $\bar h(\cdot - \lambda(\theta))$ of $\bar h$ such that
\begin{equation}\label{heq6}
 m_0(\theta) = \sum_{k \in \Z} e^{\alpha k} \bar h(k-\lambda(\theta))\,, \qquad \lambda(\theta) :=  \frac{\ln m_0(\theta)}{\alpha}\,.
\end{equation}

We can now formulate our main results.
\begin{theorem}\label{thm1}
Consider the solution $h$ of \eqref{heq1} with a nonnegative initial condition $h_0 \in L^{\infty} (\R)$ satisfying $m_0 \in L^1(0,1)$,  where $m_0$ is defined in \eqref{mnoddef}. Introducing
$$
 \mu(t,x):=\frac{\ln \big( m_0\big(\Theta(t,x)\big)\big)}{\alpha}\,, \quad (t,x)\in (0,\infty)\times\R\,,
$$
there holds 
$$
\lim_{t\to\infty} \int_{\R} e^{\alpha x} \Big|h(t,x)-\bar h\big(x-\mu(t,x)\big)\Big| \,dx = 0\,.
$$
\end{theorem}

Theorem~\ref{thm1} provides a detailed description of the asymptotic behaviour of solutions to \eqref{heq1} as $t \to \infty$ and in particular  implies
 that the long-time 
behaviour is in general periodic. It also  characterizes the class of initial data with finite mass that yield a traveling wave solution (resp. self-similar solution 
in the original variables) as $t \to \infty$.

\begin{corollary}\label{cor1}
 We have convergence to a stationary solution if and only if the function $m_0$ as defined in \eqref{mnoddef} is constant almost everywhere in $[0,1)$.
\end{corollary}

Another consequence of Theorem~\ref{thm1} is the fact that stationary solutions of \eqref{heq1} are unique up to rescaling. 

\begin{corollary}\label{cor2}
Let ${\bar h}_1$ and ${\bar h}_2$ be two nonnegative solutions to \eqref{heq3} in $L^1(\R;e^{\alpha x}\,dx)$. Then there exists $\Lambda \in \R$ such that ${\bar h}_1(x)= {\bar h}_2(x-\Lambda)$ for $x\in\R$.
\end{corollary}

The remainder of the paper is organized as follows. 
Section~\ref{S.thm1}, which is the core of our paper, contains the proof of Theorem~\ref{thm1} and consists of several lemmas. The key idea is that
we can construct a Lyapunov functional for the evolution in each fibre that resembles a contractivity property of solutions to scalar conservation laws and nonlinear diffusion equations  (see Lemma~\ref{L.lyapunov}) \cite{DE05, Vaz07}. Another key auxiliary result is a tightness property that we prove in Lemma~\ref{L.tightness}.  This in turn allows us to obtain a lower bound for the  dissipation functional in Lemma~\ref{L.dissipation}. All these results are proved pointwise in the sense that we prove them for each fibre. Corollary~\ref{cor1} is then a direct  consequence of Theorem~\ref{thm1}. Finally, in order to apply Theorem~\ref{thm1} to deduce the uniqueness result stated in Corollary~\ref{cor2} we need an \textit{a priori} $L^{\infty}$ bound for arbitrary nonnegative solutions to \eqref{heq3}. The corresponding result provided in Lemma~\ref{L.regularity} and the conclusion
are the contents of Section~\ref{S.apriori}.

\section{Proof of Theorem~\ref{thm1}}
\label{S.thm1}

Consider $h_0\in L^\infty(\R)$, $h_0\ge 0$, such that $m_0$ defined in \eqref{mnoddef} belongs to $L^1(0,1)$. In particular, $m_0(\theta)$ is finite for almost every $\theta\in [0,1)$. Let $h$ be the solution to \eqref{heq1} with initial condition $h(0)=h_0$ at time $t=0$. 

We introduce further suitable variables. For $\theta \in [0,1)$ such that $m_0(\theta) \in (0,\infty)$  we define  the right-continuous functions
\begin{equation}\label{thm13}
\p^{\theta}_{k}(t):= h\big(t,k+1-\psi(t,\theta)\big)\qquad \mbox{ and } \qquad {\bar \p}^{\theta}_{k}(t):=\bar h\big(k+1 -\lambda(\theta)-\psi(t,\theta)\big)
\end{equation}
for $(t,k)\in [0,\infty)\times\Z$, where $\psi$ is defined in \eqref{psidef} and $\lambda(\theta)=\ln m_0(\theta)/\alpha$. 
Observe that ${\bar \p}^{\theta}_k$ is $1$-periodic. 

The motivation for introducing these variables is as follows. In order to get rid of the transport term $\partial_x h$ in \eqref{heq1} we as usual go over to characteristics. However, the solution is then continuously shifted to the right by the coagulation term. 
Thus, to keep the solution at scales of order one, we shift the solution back to the left at times $n+\theta$, $n \in \N$. In a more quantitative way, we report the following identity:
\begin{lemma}\label{L.htophi}
For $t>0$ there holds 
$$
\int_\R e^{\alpha x} \Big| h(t,x) - \bar h(x-\mu(t,x)) \Big| \,dx = \int_0^1 e^{\alpha(1-\psi(t,\theta))} \sum_{k\in\Z} e^{\alpha k} \Big| \varphi_k^\theta(t) - {\bar\varphi}^\theta_k(t) \Big| \,d\theta\,.
$$
\end{lemma}
\begin{proof}
Clearly
\begin{align*}
& \int_\R e^{\alpha x} \Big| h(t,x) - \bar h(x-\mu(t,x)) \Big| \,dx \\
& = \int_0^1 e^{\alpha y} \sum_{k\in\Z} e^{\alpha k} \Big| h(t,y+k) - {\bar h}(y+k-\mu(t,y+k)) \Big| \,dy\,.
\end{align*}
We next make the change of variables $y=\theta - t + \lfloor t \rfloor$ for $y\in (t - \lfloor t \rfloor,1)$ and $y = 1+ \theta - t + \lfloor t \rfloor$ for $y\in (0,t - \lfloor t \rfloor)$ and use the definition of $\mu$ and the properties of $\Theta$ and $\psi$ to complete the proof.
\end{proof}

Define next
\begin{equation}
\mathcal{J}^\theta := (0,\theta) \cup \bigcup_{n\in\N} (n+\theta,n+1+\theta)\,. \label{phl3}
\end{equation}
It follows from \eqref{heq1} that
\begin{equation}\label{thm14}
 \frac{d}{dt}\p^{\theta}_{k}(t)= \alpha \p^{\theta}_{k}(t) + e^{-\alpha}\big(\p^{\theta}_{k-1}\big)^2(t) - \big(\p^{\theta}_{k}\big)^2(t) \qquad \mbox{ for } t \in \mathcal{J}^\theta
\end{equation}
and, by \eqref{jumppsi} and \eqref{thm13},
\begin{equation}\label{thm13b}
 \p^{\theta}_{k}\big((n+\theta)^-\big) = \lim_{t \nearrow n+\theta} \p^{\theta}_{k}(t) = h(t,n+\theta) = \lim_{t\searrow n+\theta} \p^{\theta}_{k-1}(t) =  \p^{\theta}_{k-1}\big((n+\theta)\big)\,.
\end{equation}
Also,
\begin{equation}
\sum_{k\in\Z} e^{\alpha(k+\theta)} \p_k^{\theta}(0)= m_0(\theta)\,, \label{phl4}
\end{equation} 
and, since 
$$
\frac{d}{dt} \sum_{k\in\Z} e^{\alpha k} \varphi_k^\theta(t) = \alpha \sum_{k\in\Z} e^{\alpha k} \varphi_k^\theta(t)\,, \qquad t\in \mathcal{J}^\theta\,,
$$
by \eqref{thm14}, we deduce that
$$
\sum_{k\in\Z} e^{\alpha k} \varphi_k^\theta(t) = e^{\alpha(t-\theta)} m_0(\theta)\,, \quad t\in [0,\theta)\,,  
$$
and, for $n\in\N$,
$$
\sum_{k\in\Z} e^{\alpha k} \varphi_k^\theta(t) = e^{\alpha(t-n-\theta)} \sum_{k\in\Z} e^{\alpha k} \varphi_k^\theta(n+\theta)\,, \quad t\in [n+\theta,n+1+\theta)\,.
$$
As 
$$
\sum_{k\in\Z} e^{\alpha k} \varphi_k^\theta\big( (n+\theta)^- \big) = e^\alpha \sum_{k\in\Z} e^{\alpha k} \varphi_k^\theta(n+\theta) 
$$
for $n\in\N$ by \eqref{thm13b}, an induction argument leads us to the identity
\begin{equation}\label{thm14b}
 \sum_{k\in\Z}  e^{\alpha k} \p^{\theta}_{k}(t)= e^{\alpha (\psi(t,\theta)-1)} m_0(\theta)\,, \quad t\ge 0\,.
\end{equation}

Next, thanks to \eqref{heq3}, \eqref{heq6}, and \eqref{thm13}, we easily check that $({\bar \p}^{\theta}_{k})_k$ satisfy also \eqref{thm14}, \eqref{thm13b}, and \eqref{thm14b}. Due to \eqref{S2E5} we also have that ${\bar \p}^{\theta}_{k}(t) \leq C \exp\big(-L2^k\big)$ as $k \to \infty$ uniformly for $t\in [0,\infty)$.

\bigskip

We now prove a well-posedness result for \eqref{thm14}-\eqref{thm13b} that in particular gives a uniform bound on $(\p^{\theta}_{k})_k$ that is independent of  $\theta$. Let $\mathcal{Y}$ be the space 
$$
\mathcal{Y} := \left\{ \phi = (\phi_k)_{k\in\Z}\ :\ \phi\in l_\infty(\Z) \;\mbox{ and }\; \sum_{k\in\Z} e^{\alpha k} |\phi_k| < \infty \right\}
$$
and denote its positive cone by $\mathcal{Y}_+$.

\begin{lemma}\label{L.wp}
Fix $\theta\in [0,1)$ such that $m_0(\theta)<\infty$. Let  $\p^{\theta}(0)=\Big(\p^{\theta}_{k}(0)\Big)_k \in \mathcal{Y}_+$ such that
\begin{equation} 
\left\| \p^{\theta}(0) \right\|_{l_\infty} \leq C_0 \;\mbox{ and }\; \sum_{k\in\Z} e^{\alpha k} \p_{k}^{\theta}(0)=e^{-\alpha \theta} m_0(\theta) \label{phl5}
\end{equation}
for some $C_0>0$. Then there exists a unique function $\p^{\theta}\in C^0\big([0,\infty); \mathcal{Y}_+\big)$ with $\p_k^{\theta} \in C^1\big( \mathcal{J}^\theta \big)$ for all $k\in\Z$ which solves \eqref{thm14} and \eqref{thm13b} with initial condition $\p^\theta(0)$ and satisfies \eqref{thm14b} as well as the following property:
\begin{equation}\label{wp0}
\sup_{t\ge 0} \left\| \p^{\theta}(t) \right\|_{l_\infty} \leq c_0 := \max\Big\{ C_0, \frac{\alpha}{1{-} e^{-\alpha}}\Big\}\,. 
 \end{equation}
\end{lemma}

\begin{proof}
Throughout the proof we omit the dependence on  $\theta$ in the notation. It turns out to be convenient to go over to the unknown function 
$$
\rho_k(t):=e^{-\alpha (\psi(t)-1)}\p_k(t)\,, \quad (t,k)\in [0,\infty)\times\Z\,.
$$ 
Then we are  looking for a solution of
\begin{equation}\label{wp1}
 \frac{d}{dt} \rho_k(t) = e^{\alpha (\psi(t)-1)}\Big( e^{-\alpha} \rho_{k-1}^2(t) - \rho_k^2(t) \Big)\,, \quad t\in\mathcal{J}\,,
\end{equation}
satisfying
\begin{equation}
\rho_k\big((n+\theta)^-\big) = e^{-\alpha} \rho_{k-1}(n+\theta)\,, \qquad n\in\N\,, \ k\in\Z\,, \label{phl6}
\end{equation}
and
\begin{equation}\label{wp2}
 \sum_{k\in\Z} e^{\alpha k} \rho_k(t)=m_0(\theta) \qquad \mbox{ for all } t \,.
\end{equation}
Introducing $R(\rho):=(R_k(\rho))_{k\in\Z}$ with $R_k(\rho) := e^{-\alpha} \rho_{k-1}^2 - \rho_k^2$, one has 
$$
\left\| R(\rho) - R(\sigma) \right\|_{l_1(e^{\alpha k})} \le 2 \left( \|\rho\|_{l_\infty} + \|\sigma\|_{l_\infty} \right) \left\| \rho - \sigma \right\|_{l_1(e^{\alpha k})}
$$
for any $\rho=(\rho_k)_k$ and $\sigma=(\sigma_k)$ in $\mathcal{Y}$. Consequently, a standard fixed point argument implies that there is $T\in (0,\theta)$ small enough such that there is a unique solution to \eqref{wp1} with initial condition $\left( e^{\alpha\theta} \p_k(0) \right)_k$ in
$$
\mathcal{M}_T:=\{ \rho=\big(\rho_k\big)_k \in C^0([0,T]; l_1(e^{\alpha k})\big)\ :\  \| \rho \|_{l_\infty}\leq 2c_0\}\,.
$$
Nonnegativity of each component $\rho_k$, $k\in\Z$, of the solution follows from the 
maximum principle. In order to show that $\rho$ satisfies \eqref{wp2} in $[0,T]$ we use \eqref{wp1} to compute, for $M\ge 1$, 
$$
\frac{d}{dt} \sum_{k=-M}^M e^{\alpha k} \rho_k(t) = e^{\alpha (\psi(t)-1)} \Big( e^{-\alpha(M+1)} \rho_{-(M+1)}^2(t) - e^{\alpha M}\rho_M^2(t)\Big) \,,
$$
so that
\begin{equation}
\frac{d}{dt} \sum_{k=-M}^M e^{\alpha k} \rho_k(t) \ge - e^{\alpha M}\rho_M^2(t)\,, \label{phl7}
\end{equation}
and, since $\rho\in\mathcal{M}_T$, 
$$
\frac{d}{dt} \sum_{k=-M}^M e^{\alpha k} \rho_k(t)  \leq e^{\alpha (\psi(t)-1)} e^{-\alpha(M+1)}  \rho_{-(M+1)}^2(t) \leq 4c_0^2 e^{-\alpha (M+1)}\,. 
$$
Integrating the previous inequality and letting $M \to \infty$ we obtain that
$$
\sum_{k\in\Z} e^{\alpha k} \rho_k(t) \leq m_0 \quad\mbox{ for all }\;\; t \in [0,T]\,.
$$
In particular, $\rho_M(t) \leq m_0 e^{-\alpha M}$ for $t\in [0,T]$ and we improve the lower bound \eqref{phl7} to 
$$
\frac{d}{dt} \sum_{k=-M}^M e^{\alpha k} \rho_k(t) \ge - e^{-\alpha M} m_0^2\,.
$$
Integrating the above inequality and letting $M\to\infty$ give 
$$
\sum_{k\in\Z} e^{\alpha k} \rho_k(t) \geq m_0 \quad\mbox{ for all }\;\; t \in [0,T]\,,
$$
and, as a consequence, the solution $\rho$ satisfies \eqref{wp2} in $[0,T]$.

It remains to show that \eqref{wp0} is satisfied in $[0,T]$ from which it follows that we can extend the solution to the interval $[0,\theta]$. We can then repeat the argument in all subsequent intervals $[n+\theta, n+\theta +1)$, $n\in\N$.

To prove \eqref{wp0} we use a maximum principle argument and note that, given $\gamma\in C^1([0,T])$ and $M\in\N$, it follows from \eqref{wp1} that
\begin{align*}
\frac{d}{dt} \sum_{k=-\infty}^M e^{\alpha k} \big( \rho_k - \gamma \big)_+ & = \sum_{k=-\infty}^M e^{\alpha (k-1)} \left( \rho_{k-1}^2 - \gamma^2 \right) \mathrm{sign}_+(\rho_k-\gamma) e^{\alpha(\psi-1)} \\
& \qquad + \sum_{k=-\infty}^M e^{\alpha k} \left( e^{-\alpha}-1 \right) \gamma^2 \mathrm{sign}_+(\rho_k-\gamma) e^{\alpha(\psi-1)} \\
& \qquad - \sum_{k=-\infty}^M e^{\alpha k} \left( \rho_k^2 - \gamma^2 \right)_+ e^{\alpha(\psi-1)} \\
& \qquad - \sum_{k=-\infty}^M e^{\alpha k}  \mathrm{sign}_+(\rho_k-\gamma) \frac{d\gamma}{dt} \end{align*}
Choosing $\gamma$ such that
$$
\frac{d\gamma}{dt} = -e^{\alpha(\psi-1)} (1-e^{-\alpha}) \gamma^2\,, \qquad \gamma(0)=e^{\alpha \theta} C_0\,,
$$
we realize that
\begin{align*}
\frac{d}{dt} \sum_{k=-\infty}^M e^{\alpha k} \big( \rho_k - \gamma \big)_+ & \le \sum_{k=-\infty}^{M-1} e^{\alpha k} \left( \rho_k^2 - \gamma^2 \right)_+ e^{\alpha(\psi-1)} \\
& \qquad - \sum_{k=-\infty}^M e^{\alpha k} \left( \rho_k^2 - \gamma^2 \right)_+ e^{\alpha(\psi-1)} \le 0\,,
\end{align*}
hence
$$
\sum_{k=-\infty}^M e^{\alpha k} \big( \rho_k - \gamma \big)_+(t) \le \sum_{k=-\infty}^M e^{\alpha k} \big( \rho_k - \gamma \big)_+(0)=0\,, \qquad t\in [0,T]\,.
$$
As $M\in\N$ is arbitrary we conclude that
$$
\rho_k(t)\le \gamma(t)\,, \qquad (t,k)\in [0,T]\times\Z\,.
$$
Since $\psi(t,\theta)-1=t-\theta$ for $t \in (0,\theta)$,  the function $\gamma$ is given by
$$
\gamma(t) = \frac{\alpha C_0 e^{\alpha\theta}}{\alpha  +C_0 (e^{\alpha t} -1)(1{-}e^{-\alpha})}\,, \qquad t\in [0,T]\,,
$$
and we end up with
$$
\varphi_k(t) = e^{\alpha(t-\theta)} \rho_k(t) \le e^{\alpha (t-\theta)} \gamma(t) \le \sup_{\tau\ge 0} \left\{ \frac{\alpha C_0 e^{\alpha\tau}}{\alpha + C_0 (e^{\alpha\tau} -1)(1{-}e^{-\alpha})} \right\} = c_0
$$
for $(t,k)\in [0,T]\times\Z$, which proves the claim.
\end{proof}

To proceed further, we set
\begin{equation}\label{thm15}
 L^{\theta}(t):= \sum_k e^{\alpha k} \Big(\p^{\theta}_{k}(t)- {\bar \p}^{\theta}_{k}(t)\Big)_+ \,, \qquad t\ge 0\,,
\end{equation}
which is a right-continuous function of time, and notice that \eqref{thm13} and \eqref{thm13b} give
\begin{equation}\label{thm16}
  L^{\theta}((n+\theta)^-)=  e^{\alpha} L^{\theta}((n+\theta))\,. 
\end{equation}
Our goal is to show that $L^{\theta}(t) \to 0$ as $t \to \infty$. This will follow from the fact that $L^{\theta}$ is almost a Lyapunov functional and that we can provide a lower bound on the
dissipation functional. Towards that aim we first prove a tightness result.

\begin{lemma}\label{L.tightness}
Given $\eps>0$ and $\theta\in [0,1)$ such that $m_0(\theta)<\infty$, let $M=M(\theta,\eps)>0$ be such that 
\begin{equation}
\sum_{k \geq M} e^{\alpha k} h_0(k+\theta) < \eps\,. \label{phl8}
\end{equation} 
Then there exists $N=N(\eps,M,m_0(\theta),\|h_0\|_{L^{\infty}})$ such that
\begin{equation}\label{tightness}
 \sum_{|k| \geq N} e^{\alpha k} \p^{\theta}_{k}(t) \leq 2e^{\alpha}\eps   \qquad \mbox{ for all } t\in [0,\infty)\,,
\end{equation}
and
\begin{equation}
\sum_{|k| \geq N} e^{\alpha k} {\bar \p}^{\theta}_{k}(t) \le 2e^{\alpha}\eps   \qquad \mbox{ for all } t\in [0,\infty)\,. \label{phl9}
\end{equation}
\end{lemma}

\begin{proof}
Within this proof we again omit the dependence on $\theta$ in the notation.

We first notice that the uniform bound \eqref{wp0} and the boundedness of $\bar h$ imply that there exists $N_1\in \N$ such that 
$$
\sum_{k \leq -N_1} e^{\alpha k} \p_k(t) + \sum_{k \leq -N_1} e^{\alpha k} {\bar \p}_k(t) \leq \eps\,, \qquad t\in [0,\infty)\,.
$$ 

In order to control the mass at large positive $k$ we construct a supersolution for the quantiles. More precisely, for $t\ge 0$, we define
$$
 Q_{k}(t):=  \sum_{\ell \geq k} e^{\alpha \ell} \p_{\ell}(t)
\qquad \mbox{ and } \qquad 
 {\bar Q}_{k}(t):=  \sum_{\ell \geq k} e^{\alpha \ell} {\bar \p}_{\ell-\ell_0}(t) + \eps e^{\alpha \psi(t)}
$$
with some sufficiently large $\ell_0 \in \N$ that will be determined later. Owing to \eqref{thm14} we  note that $Q_{k}$ solves
$$
 \frac{d}{dt} Q_{k}= \alpha Q_{k} + e^{-\alpha(k-1)}\Big( Q_{k-1}-Q_{k}\Big)^2 \,, \qquad t\in \mathcal{J}\,, 
$$
and ${\bar Q}_{k}$ solves the same equation for $k\in\Z$. Furthermore, by \eqref{thm13} and  \eqref{thm13b}, we have 
$$
Q_{k}((n+\theta)^-)= e^{\alpha} Q_{k-1}((n+\theta)) \;\;\mbox{ and }\;\; {\bar Q}_{k}((n+\theta)^-)= e^{\alpha} {\bar Q}_{k-1}((n+\theta))
$$
for $n\in\N$ and $k\in\Z$. We also observe that, by \eqref{thm14b}, 
$$
\lim_{k\to -\infty} Q_{k}(t) = e^{\alpha (\psi(t)-1)}m_0 \quad\mbox{ and }\quad \lim_{k\to -\infty} \bar Q_k(t) = e^{\alpha (\ell_0+\psi(t))} m_0+\eps e^{\alpha \psi(t)}\,, \quad t\ge 0\,.
$$

For the difference $W_{k}:=Q_{k} - \bar Q_{k}$ we obtain
\begin{equation}\label{Weq}
\frac{d}{dt} W_{k}= \alpha W_{k} + e^{-\alpha (k-1)} \big( W_{k-1} - W_k \big) \big( \bar Q_{k-1} - \bar Q_{k} + Q_{k-1} - Q_{k}\big)\,, \quad (t,k)\in\mathcal{J}\times\Z\,.
\end{equation}
Since
$$
\bar Q_{k-1} - \bar Q_{k} + Q_{k-1} - Q_{k} = e^{\alpha(k-1)} \bar{\p}_{k-1-\ell_0} + e^{\alpha(k-1)} \p_{k-1}\,,
$$
we infer from \eqref{thm13} and \eqref{wp0} that
\begin{equation}
0 \le \bar Q_{k-1} - \bar Q_{k} + Q_{k-1} - Q_{k} \le \left( c_0 + \|\bar h\|_{L^\infty} \right) e^{\alpha (k-1)}\,, \qquad k\in\Z\,. \label{phl10}
\end{equation} 
In addition,
$$
\lim_{k\to -\infty} W_{k}(t) = m_0 e^{\alpha(\psi(t)-1)} (1-e^{\alpha\ell_0}) - \eps e^{\alpha \psi(t)} \le - \eps\,, \qquad t\in [0,\infty)\,,
$$
so that 
$$
\sum_{k=-\infty}^K (W_k)_+ = \sum_{k=-K}^K (W_k)_+
$$
for $K\in\N$ large enough. We then deduce from \eqref{Weq} and \eqref{phl10} that
\begin{align*}
\frac{d}{dt} \sum_{k=-\infty}^K (W_k)_+ & \le \alpha \sum_{k=-\infty}^K (W_k)_+ + \sum_{k=-\infty}^K e^{-\alpha(k-1)} [ \bar Q_{k-1} - \bar Q_{k} + Q_{k-1} - Q_{k} ] (W_{k-1})_+ \\
& \le \left( \alpha + c_0 + \|\bar h\|_{L^\infty} \right) \sum_{k=-\infty}^K (W_k)_+\,,
\end{align*}
and thus, after integration,
\begin{equation}
\sum_{k=-\infty}^K (W_k)_+(t) \le \sum_{k=-\infty}^K (W_k)_+((n-1+\theta)_+)\,, \qquad t\in [(n-1+\theta)_+, n+\theta)\,, \ n\in\N\,. \label{phl11}
\end{equation}
We now choose $\ell_0=\ell_0(M,m_0)$ so large such that
$$
 \bar Q_{k}(0) \geq 2 m_0 \geq Q_{k}(0) \qquad \mbox{ for all } k \leq M\,.
$$
Also, by \eqref{phl8}, we have
$$
\bar Q_k(0) \geq \eps \geq Q_k(0) \qquad \mbox{ for all } k>M\,.
$$
With this choice of $\ell_0$, we conclude that $W_k(0)\le 0$ for all $k\in\Z$, hence $W_k(t)\le 0$ for $t\in [0,\theta)$ and $k\in\Z\cap (-\infty, K]$ by \eqref{phl11}. Since $K\in\N$ is arbitrary we have thus proved that $Q_k(t) \le \bar Q_k(t)$ for $t\in [0,\theta)$ and $k\in\Z$. In particular, 
$$
e^\alpha Q_{k-1}(\theta) = Q_k(\theta_-)\le \bar Q_k(\theta_-) = e^\alpha \bar Q_{k-1}(\theta)\,,
$$
which allows us to iterate the above argument and end up with 
$$
Q_k(t) \le \bar Q_k(t)\,, \qquad (t,k)\in [0,\infty)\times\Z\,.
$$
Now, according to \eqref{thm13}, $\bar Q_k$ is $1$-periodic and, for $t\in [0,1)$, we infer from \eqref{S2E5} that
$$
\bar Q_k(t) \le C \sum_{\ell\ge k} e^{\alpha\ell} \exp\left( - L 2^{\ell-\ell_0+1-\lambda-\psi(t)} \right) \le (2 e^{\alpha} - 1) \eps
$$
and
$$
\sum_{\ell\ge k} e^{\alpha \ell} {\bar \p}_k(t) \le C \sum_{\ell\ge k} e^{\alpha\ell} \exp\left( - L 2^{\ell+1-\lambda-\psi(t)} \right) \le (2 e^{\alpha} - 1) \eps
$$
for $k\ge N_2(\theta)$ sufficiently large. Choosing $N=\max\{N_1,N_2\}$ finishes the proof.
\end{proof}

\begin{lemma}\label{L.lyapunov}
Consider $\theta\in [0,1)$ such that $m_0(\theta)<\infty$. For $t\in\mathcal{J}^\theta$ we have
$$
\frac{d}{dt} L^{\theta}(t) =\alpha L^{\theta}(t) -D^{\theta}(t) \,,
$$
where
$$
D^{\theta} := \sum_{k\in\Z} e^{\alpha k} \big( \p^{\theta}_{k} +{\bar \p}^{\theta}_{k} \big) w^{\theta}_{k} \big[ \mathrm{sign}_+(w^{\theta}_{k}) - \mathrm{sign}_+(w^{\theta}_{k+1})\big] \geq 0\,,
$$
and $w^{\theta}_{k}:=\p^{\theta}_{k} - {\bar \p}^{\theta}_{k}$, $k\in\Z$.
\end{lemma}
 
\begin{proof}
This is a simple explicit computation.
\end{proof}

\begin{lemma}\label{L.dissipation}
Consider $\theta\in [0,1)$ such that $m_0(\theta)<\infty$ and $\eps>0$. Let $N\in\N$ be such that \eqref{tightness} holds and let $\mathcal{T}_{\eps,N}^\theta$ be the set of solutions $q=(q_k)_{k\in\Z}$ to \eqref{thm14} on $(\theta,\theta+1)$ satisfying the mass equation \eqref{thm14b}, the uniform bound \eqref{wp0}, the tightness estimate \eqref{tightness}, and the lower bound
\begin{equation}
\sum_{k\in\Z} e^{\alpha k} \big( q_k(t) - {\bar \p}_k(t))_+ \ge 16 e^\alpha \eps \label{phl12}
\end{equation}
on the interval $[\theta,\theta+1)$.

Then there exists $\nu=\nu(\theta,N,\eps)>0$ such that
$$
\int_{\theta}^{\theta+1} \sum_{k\in\Z} e^{\alpha k} \big( q_{k}(t) +{\bar \p}_{k}(t) \big) \omega_{k}(t) \big[ \mathrm{sign}_+(\omega_{k}(t)) - \mathrm{sign}_+(\omega_{k+1}(t))\big]\,dt \geq \nu\,,
$$
where $\omega_k := q_k - \bar{\p}_k^\theta$, $k\in\Z$.
\end{lemma}

\begin{proof}
Throughout this proof we again omit the dependence on $\theta$. 

Assume for contradiction that there exists a sequence $(p^m)_{m\ge 1}$ in $\mathcal{T}_{\eps,N}$ such that
\begin{equation}
\lim_{m\to\infty} \int_\theta^{\theta+1} D_m(t)\, dt = 0\,, \label{phl13}
\end{equation}
where $p^m = (p_k^m)_{k \in \Z}$, $W_k^m := p_k^m - {\bar \p}_k$, and
$$
D_m := \sum_{k\in\Z} e^{\alpha k} \big( p^m_{k} +{\bar \p}_{k} \big) W^m_{k} \big[ \mathrm{sign}_+(W^m_{k}) - \mathrm{sign}_+(W^m_{k+1})\big]\,.
$$
Owing to \eqref{thm14} and \eqref{wp0}, $(p_k^m)_{m\ge 1}$ is bounded in $C^1([\theta,\theta+1))$ for all $k\in\Z$ and we can extract a subsequence, again denoted by $(p^m)_{m\ge 1}$ such that $(p_k^m)_{m\ge 1}$ converges uniformly in $[\theta,\theta+1)$ to a function $p_k$ for all $k\in\Z$. Setting $p:=(p_k)_{k\in\Z}$, we easily see that $p$ is a solution to \eqref{thm14} and satisfies the uniform bound \eqref{wp0} and the tightness bound \eqref{tightness} in $[\theta,\theta+1)$. In addition, due to the tightness property \eqref{tightness}, we deduce from the mass equation \eqref{thm14b} and the lower bound \eqref{phl12} which are valid for $p^m$ that
\begin{equation}\label{west1}
\Big| \sum_{k\in\Z} e^{\alpha k} p_k(t) - e^{\alpha (\psi(t)-1)} m_0\Big| \leq  4e^{\alpha} \eps\, \qquad \mbox{ and } \qquad 
\Big| \sum_{|k| \leq N} e^{\alpha k} W_k(t)\Big| \leq 4e^{\alpha}\eps\,,
\end{equation}
with $W_k := p_k - {\bar \p}_k$, as well as 
\begin{equation}\label{west2}
 \sum_{|k| \leq N} e^{\alpha k} \big(W_k\big)_+(t) \geq 8 e^{\alpha} \eps\,,
\end{equation}
for $t\in [\theta,\theta+1)$. 

Now, since each term in the sum $D_m$ is non-negative, there holds
\begin{equation}\label{limitdiss}
 \lim_{m \to \infty} \int_{\theta}^{\theta+1} b_k^m(t) \,dt =0\,, \quad k\in\Z\,,
\end{equation}
with
\begin{equation}\label{Bdef}
b_k^m := e^{\alpha k} \big( p_k^m+{\bar \p}_k\big) W^m_k \big( \mathrm{sign}_+(W^m_k) - \mathrm{sign}_+(W^m_{k+1})  \big)\,, \quad k\in\Z\,.
\end{equation}
It  remains to take the limit in \eqref{limitdiss}. Fix $k\in \Z\cap [-N,N]$. Introducing
$$
b_k := e^{\alpha k} \big( p_k +{\bar \p}_k\big) W_k \big[ \mathrm{sign}_+(W_k) - \mathrm{sign}_+(W_{k+1})  \big]\,, \quad k\in\Z\,,
$$
and the sets
\begin{align*}
\mathcal{B}_k & := \left\{ t\in (\theta,\theta+1)\ :\ W_k(t)>0 \;\mbox{ and }\; W_{k+1}(t)=0 \right\}\,, \\
\mathcal{C}_k & := \left\{ t\in (\theta,\theta+1)\ :\ W_k(t)<0 \;\mbox{ and }\; W_{k+1}(t)=0 \right\}\,, \\
\mathcal{G}_k & := (\theta,\theta+1)\setminus (\mathcal{B}_k \cup \mathcal{C}_k)\,, 
\end{align*}
we have $(\theta,\theta+1) = \mathcal{B}_k \cup \mathcal{C}_k \cup \mathcal{G}_k$ and we notice that 
$$
\lim_{m\to\infty} b_k^m(t) = b_k(t) \quad\mbox{ for }\;\; t\in \mathcal{G}_k\,,
$$
while
$$
b_k(t) \le \liminf_{m\to\infty} b_k^m(t) \quad\mbox{ for }\;\; t\in \mathcal{C}_k\,.
$$
Then, \eqref{limitdiss} implies that
\begin{equation}
0= \int_{\mathcal{C}_k\cup \mathcal{G}_{k}}
B_{k}(t) \,dt + \lim_{m \to \infty} \int_{\mathcal{B}_{k}} B_{k}^{m}(t)\,dt\,. \label{S8E5}
\end{equation}
 
We now claim that $\left\vert \mathcal{B}_{k}\right\vert =0$. To prove this we use the equation satisfied by $W_{k+1}$ which reads, due to \eqref{thm14},
\begin{equation}\label{S8E6}
\frac{d}{dt}W_{k+1}(t) =\alpha W_{k+1}(t) + e^{-\alpha}\left[ p_{k}(t) + \bar{\p}_{k}(t)  \right]  W_{k}(t) - \left[  p_{k+1}(t) + \bar{\p}_{k+1}(t) \right] W_{k+1}(t)
\end{equation}
for $t\in (\theta,\theta+1)$. Let $t_0\in \mathcal{B}_k$. By \eqref{S8E6}
$$
\frac{d}{dt}W_{k+1}(t_0) = e^{-\alpha}\left[ p_{k}(t_0) + \bar{\p}_{k}(t_0)  \right]  W_{k}(t_0)>0\,.
$$
Consequently, there is $\delta>0$ such that $\mathcal{B}_k \cap [t_0-\delta,t_0+\delta] =\{t_0\}$ and $\mathcal{B}_k$ contains only isolated points. This implies that $\left\vert \mathcal{B}_{k}\right\vert =0$ and therefore \eqref{S8E5} reduces to
\begin{equation}
0=\int_{\theta}^{\theta+1} b_{k}(t)\,dt\,. \label{S8E7}
\end{equation}
This in turn implies that
\begin{equation}
W_k \left[ \mathrm{sign}_+(W_k) - \mathrm{sign}_+(W_{k+1}) \right] = 0 \quad\mbox{ a.e. in }\; (\theta,\theta+1)\,, \quad k\in\Z\cap [-N,N]\,. \label{S8E7b}
\end{equation}

We will now show that the functions $\p_k$ and the corresponding $W_k$ constructed above cannot exist. Indeed, assume first that there is $t_* \in [\theta,\theta+1)$ such that $W_{-N}(t_*)  >0$. Due to \eqref{west1} and \eqref{west2}, $W_k(t_*)$ cannot be positive for all $k\in\Z \cap [-N,N]$ and we define
$$
k_*+1 := \min\{ k\in\Z \cap [-N,N]\ : \ W_k(t_*)\le 0 \} \le N\,.
$$
Then $W_{k_*}(t_*)>0$ and $W_{k_*+1}(t_*)\le 0$. Since 
$$
\frac{d}{dt}W_{k_*+1} =\alpha W_{k_*+1} + e^{-\alpha}\left[ p_{k_*} + \bar{\p}_{k_*}  \right]  W_{k_*} - \left[  p_{k_*+1} + \bar{\p}_{k_*+1} \right] W_{k_*+1}
$$
by \eqref{S8E6}, we realize that
$$
\frac{d}{dt} \left[ W_{k_*+1}(t) e^{-\alpha t} \right] \Big|_{t=t_*} > 0\,,
$$
so that 
$$
W_{k_*+1}(t) e^{-\alpha t} < W_{k_*+1}(t_*) e^{-\alpha t_*} \le 0\,, \quad t\in [t_*-\delta,t_*)\,,
$$
for some $\delta>0$. Reducing $\delta$ if necessary, we may also assume that $W_{k_*}(t)\ge W_{k_*}(t_*)/2$ for $t\in  [t_*-\delta,t_*)$. This implies that 
$$
W_{k_*}(t) \left[ \mathrm{sign}_+(W_{k_*}(t)) - \mathrm{sign}_+(W_{k_*+1}(t)) \right] = W_{k_*}(t) \ge \frac{W_{k_*}(t_*)}{2}>0
$$
for $t\in [t_*-\delta,t_*)$ and contradicts \eqref{S8E7b}. 

Therefore $W_{-N}(t)  \leq 0$ for all $t \in [\theta,\theta+1)$. We fix any value $t_*\in [\theta,\theta+1)$ and define 
$$
k_1 := \min\{ k\in \Z\cap [-N,N]\ :\ W_k(t_*)\ge 0 \}\,.
$$
Then $k_1>-N$ and \eqref{west2} guarantees that $k_1\le N$. Since $W_{k_1-1}(t_*)<0$ and $W_{k_1}(t_*)\ge 0$, we use once more \eqref{S8E6} to obtain
$$
\frac{d}{dt} \left[ W_{k_1}(t) e^{-\alpha t} \right] \Big|_{t=t_*} < 0\,.
$$
We may then find $\delta>0$ sufficiently small such that
$$
W_{k_1-1}(t) < \frac{W_{k_1-1}(t_*)}{2} < 0 < W_{k_1}(t)\,, \quad t\in [t_*-\delta,t_*)\,.
$$
This implies that 
$$
W_{k_1-1}(t) \left[ \mathrm{sign}_+(W_{k_1-1}(t)) - \mathrm{sign}_+(W_{k_1}(t)) \right] = - W_{k_1-1}(t) \ge - \frac{W_{k_1-1}(t_*)}{2}>0
$$
for $t\in [t_*-\delta,t_*)$ and contradicts again \eqref{S8E7b}. This concludes the proof.
\end{proof}
  
\begin{lemma}\label{L.main}
Let $\theta\in [0,1)$ be such that $m_0(\theta)<\infty$. The function $L^{\theta}$ defined in \eqref{thm15} satisfies
$$
\lim_{t\to\infty} L^{\theta}(t) = \lim_{t\to\infty} \sum_{k\in\Z} e^{\alpha k} \left| \p_k^\theta(t) - \bar{\p}_k^\theta(t) \right| = 0\,.
$$
\end{lemma}

\begin{proof}
We first note that, for $t\in\mathcal{J}^\theta$, Lemma~\ref{L.lyapunov} ensures that
\begin{equation}
\frac{d}{dt} \left( L^\theta(t) e^{-\alpha t} \right) = - e^{\alpha t} D^\theta(t)\,. \label{phl50}
\end{equation}
After integration we find, for $n\in\N$,
$$
L^\theta((n+1+\theta)^-) + \int_{n+\theta}^{n+1+\theta} D^\theta(s)\,ds \le e^\alpha L^\theta(n+\theta)\,,
$$
which, together with \eqref{thm16}, gives
\begin{equation}
e^\alpha L^\theta((n+1+\theta)) + \int_{n+\theta}^{n+1+\theta} D^\theta(s)\,ds \le e^\alpha L^\theta(n+\theta)\,. \label{phl51}
\end{equation}
In particular, $(L^\theta(n+\theta))_{n\in\N}$ is a nonincreasing sequence and
\begin{equation}
\sum_{n\in\N} \int_{n+\theta}^{n+1+\theta} D^\theta(t)\,dt \le e^\alpha L^\theta(\theta)\,. \label{phl52}
\end{equation}

Assume for contradiction that
$$
\eta := \inf_{n\in\N} L^\theta(n+\theta) > 0\,.
$$
Then, owing to \eqref{thm16}, $L^\theta((n+\theta)^-)\ge e^\alpha \eta$ while we infer from \eqref{phl50} that
\begin{align*}
L^\theta(\theta^-) e^{\alpha(t-\theta)} & \le L^\theta(t) \,, \qquad t\in [0,\theta)\,, \\
L^\theta((n+1+\theta)^-) e^{\alpha(t-n-1-\theta)} & \le L^\theta(t) \,, \qquad t\in [n+\theta,n+1+\theta)\,,
\end{align*}
 for $n\in\N$. Combining these estimates gives
\begin{equation}
L^\theta(t)\ge \eta\,, \qquad t\in [0,\infty)\,. \label{phl53}
\end{equation} 
 
Consider now $\eps\in (0,\eta/(16e^\alpha))$. Since $m_0(\theta)<\infty$ there is $M=M(\theta)\in\N$ such that \eqref{phl8} holds true and it follows from Lemma~\ref{L.tightness} that there is $N=N(\theta,\eps)$ such that
\begin{equation}
\sum_{|k|\ge N} e^{\alpha k} \p_k^\theta(t) \le 2 e^\alpha \eps\,, \qquad t\in [0,\infty)\,. \label{phl54}
\end{equation} 
We next define 
$$
p_k^m(s) := \p_k^\theta(m+s)\,, \qquad (s,m,k)\in [\theta,\theta+1)\times\N\times\Z\,.
$$
Then, for each $m\in\N$, $p^m:=(p_k^m)_{k\in\Z}$ is a solution to \eqref{thm14} in $(\theta,\theta+1)$ which satisfies the mass equation \eqref{thm14b} in $[\theta,\theta+1)$. It also enjoys the uniform estimate \eqref{wp0} by Lemma~\ref{L.wp} as well as the tightness property \eqref{tightness} by \eqref{phl54}, still in $[\theta,\theta+1)$. Finally, for $t\in [\theta,\theta+1)$, 
$$
\sum_{k\in \Z} e^{\alpha k} \left( p_k^m(t) - {\bar \p}_k^\theta(t) \right)_+ = L^\theta(m+s) \ge \eta \ge 16 e^\alpha \eps
$$
by \eqref{phl53}. In other words, $p^m$ belongs to the set $\mathcal{T}_{\eps,N}^\theta$ defined in Lemma~\ref{L.dissipation} and we infer from Lemma~\ref{L.dissipation} that there is $\nu=\nu(\eps,\theta)>0$ such that 
$$
\int_\theta^{\theta+1} \sum_{k\in\Z} e^{\alpha k} \left( p_k^m(t) + \bar{\p}_k^\theta(t) \right) W_k^m(t) \left[ \mathrm{sign}_+(W_k^m(t)) - \mathrm{sign}_+(W_{k+1}^m(t))\right]\,dt \ge \nu\,, 
$$
where $W_k^m := p_k^m - \bar{\p}_k^\theta$, $(m,k)\in \N\times\Z$. In terms of $\varphi^\theta$, 
$$
\int_{m+\theta}^{m+1+\theta} D^\theta(t)\,dt \ge \nu\,, \qquad m\in\N\,.
$$
Summing up with respect to $m\in\N$ clearly contradicts \eqref{phl52}. Therefore $\eta=0$ and, since
$$
L^\theta(t) \le L^\theta(n+\theta) e^\alpha\,, \qquad t\in [n+\theta,n+1+\theta)\,, \ n\in\N\,,
$$
by \eqref{phl50}, we conclude that $L^\theta(t)\to 0$ as $t\to\infty$. 

Finally, let $t\ge 0$. We recall that
$$
\sum_{k\in\Z} e^{\alpha k} \p_k^\theta(t) = m_0(\theta) e^{\alpha(\psi(t,\theta)-1)} = \sum_{k\in\Z} e^{\alpha k} \bar{\p}_k^\theta(t)
$$
by \eqref{thm14b}, so that
\begin{align*}
\sum_{k\in\Z} e^{\alpha k} \left| \p_k^\theta(t) - \bar{\p}_k^\theta(t) \right| & = \sum_{k\in\Z} e^{\alpha k} \left[ 2 \left( \p_k^\theta(t) - \bar{\p}_k^\theta(t) \right)_+ - \left( \p_k^\theta(t) - \bar{\p}_k^\theta(t) \right) \right] \\
& = 2 L^\theta(t)\,,
\end{align*}
and the proof is complete.  
\end{proof}

\begin{proof}[{\bf Proof of Theorem~\ref{thm1}.}]
By Lemma~\ref{L.htophi}  
$$
\int_\R e^{\alpha x} \Big| h(t,x) - \bar h(x-\mu(t,x)) \Big| \,dx = \int_0^1 e^{\alpha(1-\psi(t,\theta))} \sum_{k\in\Z} e^{\alpha k} \Big| \varphi_k^\theta(t) - {\bar\varphi}^\theta_k(t) \Big| \,d\theta\,.
$$
On the one hand, since $m_0\in L^1(0,1)$, then $m_0(\theta)$ is finite for almost every $\theta\in (0,1)$ and, according to Lemma~\ref{L.main}, 
$$
\lim_{t\to\infty} \sum_{k\in\Z} e^{\alpha k} \left| \p_k^\theta(t) - \bar{\p}_k^\theta(t) \right| = 0
$$
for almost every $\theta\in (0,1)$. On the other hand, thanks to \eqref{thm14b},
$$
e^{\alpha(1-\psi(t,\theta))} \sum_{k\in\Z} e^{\alpha k} \Big| \varphi_k^\theta(t) - {\bar\varphi}^\theta_k(t) \Big| \le 2 m_0(\theta)\,.
$$ 
The Lebesgue dominated convergence then entails that 
$$
\lim_{t\to\infty} \int_\R e^{\alpha x} \Big| h(t,x) - \bar h(x-\mu(t,x)) \Big| \,dx = 0\,,
$$
which proves the claim.
\end{proof}

\section{A priori estimates for stationary solutions}
\label{S.apriori}

\begin{lemma}\label{L.regularity}
Let $\bar h \in C^{1}(\R)\cap L^1(\R;e^{\alpha x}\,dx)$ be a nonnegative solution to \eqref{heq3}. Then $\bar h \in L^{\infty}(\R)$.
\end{lemma}

\begin{proof}
 {\it Step 1:} We first claim that the  function $\bar h$ satisfies
\begin{equation}\label{reg1}
e^{\alpha x} \bar h(x) =  \int_{x-1}^x {\bar h}^2(t) e^{\alpha t}\,dt \qquad \mbox{ for all } x \in \R\,.
\end{equation}
To prove \eqref{reg1}, we first multiply \eqref{heq3} by $e^{\alpha x}$ and integrate to find
\begin{equation}\label{reg2}
e^{\alpha x} \bar h(x) =  \int_{x-1}^x {\bar h}^2(t) e^{\alpha t}\,dt +J 
\end{equation}
 for some $J \in \R$. In particular, $e^{\alpha x} \bar{h}(x) \ge J$ and the integrability of $x\mapsto e^{\alpha x} \bar h(x)$ implies that $J\le 0$. 
 
It next follows from the integrability of $x\mapsto e^{\alpha x} \bar h(x)$ that there is a sequence $(x_n)_{n\ge 1}$ in $(1,\infty)$ such that 
\begin{equation}
\lim_{n\to\infty} x_n = \infty \;\;\mbox{ and }\;\; \bar h(x_n - 1) \le e^{\alpha (x_n-1)} \bar h(x_n - 1) \le \frac{1}{2}\,, \quad n\ge 1\,. \label{phl100}
\end{equation}
We are going to show that the differential equation \eqref{heq3}  for $\bar h$ guarantees that $\bar h$ does not exceed one in $(x_n-1,x_n)$. Indeed, by \eqref{heq3}, $\partial_x \bar h \leq {\bar h}^2$ and integrating this equation gives, thanks to \eqref{phl100}, 
$$
2 - \frac{1}{\bar h(x)} \le \frac{1}{\bar h(x_n-1)} - \frac{1}{\bar h(x)} \le x - x_n + 1 \le 1\,, \quad x\in [x_n-1,x_n]\,,
$$
hence
\begin{equation}
\bar h(x) \le 1 \,, \quad x\in [x_n-1,x_n]\,.\label{phl101}
\end{equation}
Combining \eqref{reg2} (with $x=x_n$) and \eqref{phl101} gives
$$
0 \le - J = \int_{x_n-1}^{x_n} \bar h^2(y) e^{\alpha y}\,dy - e^{\alpha x_n} \bar h(x_n) \le \int_{x_n-1}^{x_n} \bar h(y) e^{\alpha y}\,dy\,.
$$
Since the right-hand side of the above inequality converges to zero as $n\to\infty$, we conclude that $J=0$ which proves \eqref{reg1}.

{\it Step 2:} To complete the proof we argue by contradiction and assume that there is $M\ge 8 e^{2\alpha}$ and $x_0\in\R$ such that
\begin{equation}
\bar{h}(x_0) \ge M\,. \label{phl102}
\end{equation}
According to \eqref{reg1} we have either
\begin{equation}\label{reg7}
 \int_{x_0-1/2}^{x_0} {\bar h}^2(y) e^{\alpha y}\,dy \geq \frac{\bar h(x_0) e^{\alpha x_0}}{2}
\end{equation}
or
\begin{equation}\label{reg8}
 \int_{x_0-1}^{x_0-1/2} {\bar h}^2(y) e^{\alpha y}\,dy \geq \frac{\bar h(x_0) e^{\alpha x_0}}{2}\,.
\end{equation}
In case that \eqref{reg7} holds, consider $y \in [x_0, x_0+1/2]$. Then $(x_0-1/2,x_0)\subset (y-1,y)$ and we deduce from \eqref{reg1}, \eqref{phl102}, and \eqref{reg7} that 
$$
\bar h(y) =  e^{-\alpha y} \int_{y-1}^y {\bar h}^2(z) e^{\alpha z}\,dz \geq  e^{-\alpha y} \int_{x_0-1/2}^{x_0}{\bar h}^2(z) e^{\alpha z}\,dz  \geq
 \frac{M}{2e^{\alpha/2}}\,.
$$
Combining this lower bound with \eqref{reg1} gives
$$
\bar h\left( x_0+\frac{1}{2} \right) \ge e^{-\alpha(x_0+1/2)} \int_{x_0}^{x_0+1/2} \bar{h}^2(y) e^{\alpha y}\, dy \ge \frac{M^2}{8 e^{2\alpha}} \ge  M \,.
$$
Similarly, if \eqref{reg8} holds, then $(x_0-1,x_0-1/2)\subset (y-1,y)$ for $y \in [x_0-1/2,x_0]$ and it follows from \eqref{reg1}, \eqref{phl102}, and \eqref{reg8} that $\bar h(y) \geq M/(2 e^{\alpha/2})$. In particular, we find
$$
\bar h\left( x_0+\frac{1}{2} \right) \ge e^{-\alpha(x_0+1/2)} \int_{x_0-1/2}^{x_0} {\bar h}^2(z)e^{\alpha z}\,dz \ge \frac{M^2 }{8 e^{2\alpha}} \ge M\,.
$$
We have thus proved that
\begin{equation}\label{reg9}
\bar h\left( x_0+\frac{1}{2} \right) \ge M\,.
\end{equation}
Arguing by induction, we actually conclude that
\begin{equation}\label{reg10}
\bar h\left( x_0+\frac{k}{2} \right) \ge M\,, \qquad k\in\N\,.
\end{equation}

Furthermore, arguing as in the proof of \eqref{reg9}, we infer from \eqref{reg1} and \eqref{reg10} that 
$$
\bar h(y) \ge \frac{M}{2e^{\alpha/2}}\,, \qquad y\in I_k\,,
$$
where either $I_k := (x_0+(k-1)/2,x_0+k/2)$ or $I_k := (x_0+k/2,x_0+(k+1)/2)$. Define
$$
\mathcal{L} := \left\{ k\in\N\ :\ I_k = \left( x_0 + \frac{k-1}{2} , x_0 + \frac{k}{2} \right) \right\}
$$
and
$$
\mathcal{R} := \left\{ k\in\N\ :\ I_k = \left( x_0 + \frac{k}{2} , x_0 + \frac{k+1}{2} \right) \right\}\,.
$$
Then $\mathcal{L} \cup \mathcal{R} = \N$ and 
\begin{align*}
1 = \int_{\R} e^{\alpha x} \bar{h}(x)\,dx & \ge \frac{1}{2} \sum_{k\in\N} \int_{x_0+(k-1)/2}^{x_0+k/2} e^{\alpha x} \bar{h}(x)\,dx + \frac{1}{2} \sum_{k\in\N} \int_{x_0+k/2}^{x_0+(k+1)/2} e^{\alpha x} \bar{h}(x)\,dx \\
& \ge \frac{1}{2} \sum_{k\in\mathcal{L}} \int_{I_k} e^{\alpha x} \bar{h}(x)\,dx + \frac{1}{2} \sum_{k\in\mathcal{R}} \int_{I_k} e^{\alpha x} \bar{h}(x)\,dx \\
& \ge \frac{M}{8e^{\alpha/2}} \left[ \sum_{k\in\mathcal{L}} e^{\alpha (x_0+(k-1)/2)} + \sum_{k\in\mathcal{L}} e^{\alpha (x_0+k/2)} \right]\,.
\end{align*}
Since $\mathcal{L}$ or $\mathcal{R}$ is infinite, the right-hand side of the above inequality cannot be bounded, which gives a contradiction and hence proves the statement of the lemma.
\end{proof}

\bigskip

\begin{proof}[Proof of Corollary~\ref{cor2}]
Let $\tilde h \in C^1(\R)\cap L^1(\R;e^{\alpha x}\,dx)$ be a nonnegative solution to \eqref{heq3}. By Lemma~\ref{L.regularity} we have $\tilde h \in L^{\infty}(\R)$. Then $\tilde h$ is also a solution to the evolution equation \eqref{heq1} and, by Theorem~\ref{thm1},
$$
\lim_{t\to\infty} \int_{\R} e^{\alpha x} \Big| \tilde h(x) - \bar h \big( x - \tilde{\mu}(t,x)\big) \Big| = 0\,,
$$
where
$$
\tilde{\mu}(t,x) := \frac{\ln\left( \tilde{m}_0(\Theta(t,x)) \right)}{\alpha}\,, \qquad (t,x)\in [0,\infty)\times\R\,,
$$
with $\Theta$ defined in Definition~\ref{theta} and
$$
\tilde{m}_0(\theta) := \sum_{k\in\Z} e^{\alpha(k+\theta)} \tilde{h}(k+\theta)\,, \qquad \theta\in [0,1)\,.
$$
Arguing as in the proof of \eqref{phl1}, there exists $\tilde{\nu}\in [0,\infty)$ such that $\tilde{m}_0(\theta)=\tilde{\nu}$ for a.e. $\theta\in (0,1)$, that is, $\tilde{\mu}$ is a constant and the claim is proved.
\end{proof}

  \bigskip
  {\bf Acknowledgment.}
  The authors gratefully acknowledge support through the CRC {\it The mathematics of emergent effects} at the University of Bonn which is funded by
  the German Science Foundation (DFG).

\bigskip

{\small
\bibliographystyle{alpha}%

}

\end{document}